\documentclass[12pt]{article}
\usepackage{amsmath,amssymb,amsthm}

\newtheorem*{lemma}{Lemma} 

\begin{document}

\date{\small Mathematics Subject Classification: 11B65, 33B99}

\title{A Simple Proof of Schmidt's Conjecture}

\author{Thotsaporn ``Aek'' Thanatipanonda
\thanks{supported by the strategic program ``Innovatives O\"{O} 2010 plus"
by the Upper Austrian Government}\\
\small{Research Institute for Symbolic Computation (RISC)}\\[-0.8ex]
\small{Johannes Kepler University, A-4040 Linz, Austria}\\
\small{\tt thotsaporn@gmail.com}%
}

\maketitle

\begin{abstract}
Using Difference Equations and Zeilberger's algorithm, 
we give a very simple proof of a conjecture of Asmus Schmidt 
that was first proved by Zudilin.

\end{abstract} %

For any integer $r \geq 1$, the sequence of numbers $\{{c^{(r)}_{k}}\}_{k \geq 0} $
is defined implicitly by
\[
\sum_k\binom{n}{k}^r\binom{n+k}{k}^r =
\sum_k\binom{n}{k}\binom{n+k}{k}c^{(r)}_k,\quad n=0,1,2,\dots
\]
In 1992, Asmus Schmidt \cite{Schmidt} conjectured 
that all $c^{(r)}_k$ are integers.
In Concrete Mathematics \cite{Knuth} on page 256,
it was stated as a research problem.
Already here, it was indicated that H. Wilf had shown
the integrality of $c^{(r)}_n$ for any $r$ but only
for $n \leq 9$.
For the first nontrivial case,
$r=2$; $\sum_k \binom{n}{k}^2\binom{n+k}{k}^2$
are the famous Ap\'ery numbers,
the denominators of rational approximations to $\zeta(3)$.
This case was proved in 1992 independently by
Schmidt himself \cite{Schmidt1} and by Strehl~\cite{Strehl}.
They both gave an explicit expression for $c^{(2)}_n$
\[
    c^{(2)}_n = \sum_j
\binom{n}{j}^3
    = \sum_j
\binom{n}{j}^2\binom{2j}{n}.
\]
These numbers are called Franel numbers.
In the same paper \cite{Strehl}, Strehl also gave a proof
for $r=3$ which uses Zeilberger's algorithm
of creative telescoping.
He also gave an explicit expression for $c^{(3)}_n$
\[
    c^{(3)}_n = \sum_j
\binom{n}{j}^2\binom{2j}{j}^2\binom{2j}{n-j}.
\]
The first full proof was given by Zudilin \cite{Zudilin}
in 2004 using a multiple generalization
of Whipple's transformation for hypergeometric
functions.
Since then, the congruence properties related to
the Schmidt numbers $S^{(r)}_n := \sum_{k}\binom{n}{k}^r\binom{n+k}{k}^r$
and to the Schmidt polynomials $S^{(r)}_n(x) :=
\sum_{k}\binom{n}{k}^r\binom{n+k}{k}^{r}x^k$
have been studied extensively.
In this note, we return to Schmidt's original problem
and present a simple proof.
\medskip

It is a natural first step to investigate the
individual term $\binom{n}{k}^r\binom{n+k}{k}^r$
before considering the full sum
$\sum_k\binom{n}{k}^r\binom{n+k}{k}^r$.
Our proof rests on the following lemma,
which was proved by Guo and Zeng \cite{Guo2}.
In order to keep this note self-contained,
we give a simple, well motivated,
computer proof of their lemma.

\begin{lemma}
For $k \geq 0$ and $r \geq 1$,
there exist integers $a^{(r)}_{k,j}$
with $a^{(r)}_{k,j}=0$ for $j < k$ or $j > rk$, and
\begin{equation} \label{main}
\displaystyle\binom{n}{k}^r\binom{n+k}{k}^r
= \sum_j a^{(r)}_{k,j}
\displaystyle\binom{n}{j}\binom{n+j}{j}
\end{equation}
for all $n\geq0$.
\end{lemma}

\begin{proof}
Define $\bar{a}^{(r)}_{k,j}$ recursively by
$\bar a^{(1)}_{k,k}=1$, $\bar a^{(1)}_{k,j}=0$ ($j\neq k$) and
\begin{equation} \label{magic}
\displaystyle \bar{a}^{(r+1)}_{k,j}
= \sum_{i}
\binom{k+i}{i}\binom{k}{j-i}\binom{j}{k}
\bar{a}^{(r)}_{k,i}.
\end{equation}
Then it is clear that $\bar{a}^{(r)}_{k,j}$ are integers.

We show by induction on $r$ that $\bar{a}^{(r)}_{k,j}$ satisfies \eqref{main}.
The statement is clearly true for $r=1$.
Suppose the statement is true for $r$. Then
\begin{align*}
 \sum_{j}\bar{a}^{(r+1)}_{k,j}
            \binom{n}{j}\binom{n+j}{j}
          & = \sum_{j}\sum_{i}
                \bar{a}^{(r)}_{k,i}
                \binom{k+i}{i}\binom{k}{j-i}\binom{j}{k}
                \binom{n}{j}\binom{n+j}{j}
                \\
          &\kern150pt\text{\small (by definition of $\bar{a}^{(r+1)}_{k,j}$)} \\
          & = \sum_{i}\bar{a}^{(r)}_{k,i}\sum_{j}
                \binom{k+i}{i}\binom{k}{j-i}\binom{j}{k}
                \binom{n}{j}\binom{n+j}{j} \\
          & = \sum_{i}\bar{a}^{(r)}_{k,i}
            \binom{n}{i}\binom{n+i}{i}
            \binom{n}{k}\binom{n+k}{k}   \\
          &  = \binom{n}{k}^{r}\binom{n+k}{k}^{r}\binom{n}{k}\binom{n+k}{k}\\
          &\kern150pt\text{\small (by induction hypothesis)}  \\
          &  = \binom{n}{k}^{r+1}\binom{n+k}{k}^{r+1}. \\
\end{align*}
The identity from line 2 to line 3,
\[
\binom{n}{i}\binom{n+i}{i}
\binom{n}{k}\binom{n+k}{k}
= \sum_{j} \binom{k+i}{i}\binom{k}{j-i}\binom{j}{k}
\binom{n}{j}\binom{n+j}{j},
\]
can be verified easily with Zeilberger's algorithm.

Therefore $\bar{a}^{(r)}_{k,j}$ satisfies \eqref{main}.
For the lemma, we can now take $a^{(r)}_{k,j} = \bar{a}^{(r)}_{k,j}$.
\end{proof}

The definition \eqref{magic} may seem to come out
of nowhere. It was found as follows.
We tried to find a relation of the form:
\[
a^{(r+1)}_{k,j}
= \sum_{i} s(k,j,i)
a^{(r)}_{k,i}.
\]
with the hope to find a nice formula for $s(k,j,i)$,
free of $r$.
The coefficients $s(k,j,i)$ then were found
by automated guessing.
First we calculated the numbers $a^{(r)}_{k,j}$
for $r$ from $1$ to $15$ and all $k,j$. Then we made
an ansatz for a hypergeometric term $s(k,j,i)$.
Fitting this ansatz to the
calculated data and solving the
constants led to the conjecture
\[
s(k,j,i) = \displaystyle\binom{k+i}{i}\binom{k}{j-i}\binom{j}{k}.
\]

Now we give a proof of the main statement.
By the lemma, we have
\begin{align*}
\sum_i \binom{n}{i}^r\binom{n+i}{i}^r
          & = \sum_i\sum_k a^{(r)}_{i,k}
          \binom{n}{k}\binom{n+k}{k}
           = \sum_k
          \binom{n}{k}\binom{n+k}{k}\sum_i a^{(r)}_{i,k}.
\end{align*}
Therefore, we have
\[
c^{(r)}_k = \sum_i a^{(r)}_{i,k}.
\]
which concludes our statement.
\subsection*{Acknowledgement}
I want to thank Veronika Pillwein and Manuel Kauers
for their helpful suggestions and support.

\end{document}